\documentclass[reqno]{amsart}

\usepackage{amsmath,amssymb,amsthm}

\usepackage{tikz}
\usepackage{caption}
\usepackage{graphicx}
\usepackage{graphics}

\newtheorem*{thm}{Theorem}
\newtheorem*{lemma}{Lemma}

\newcommand{\diam}{\operatorname{diam}}

\begin{document}

\title[]{Many antipodes implies many neighbors}

\author[]{Stefan Steinerberger}
\address{Department of Mathematics, University of Washington, Seattle}
\email{steinerb@uw.edu}

\begin{abstract} Suppose $\left\{x_1, \dots, x_n\right\} \subset \mathbb{R}^2$ is a set of $n$ points in the plane with diameter $\leq 1$, meaning $\|x_i - x_j\| \leq 1$ for all $1 \leq i,j \leq n$. We show that if there are many `antipodes', these are pairs of points of with distance $\geq 1-\varepsilon$, then there are many neighbors, these are pairs of points that are distance $\leq \varepsilon$. More precisely, we prove that for some universal $c>0$,
$$ \# \left\{(i,j): \|x_i - x_j\| \leq \varepsilon\right\} \geq 
\frac{c \cdot \varepsilon^{3/4}}{\left( \log \varepsilon^{-1} \right)^{1/4}}\cdot \# \left\{(i,j): \|x_i - x_j\| \geq 1- \varepsilon\right\}.$$
The inequality is very easy too prove with factor $\varepsilon^2$ and easy with $\varepsilon$. The optimal rate might be $\varepsilon^{1/2}$ which is attained by several examples.
\end{abstract}

 \maketitle
 
\vspace{-0pt}
\section{Introduction and Statement}
\subsection{Introduction} Suppose $\left\{x_1, \dots, x_n\right\} \subset \mathbb{R}^2$ is a set of $n$ points in the plane. A set can have `many' antipodes, these are pairs of points $(x_i, x_j)$ such that $\|x_i - x_j\|$ is close to the diameter of the set (the largest distance between any pair of points). We assume without loss of generality that the diameter is 1 and that `antipodes' are defined by having distance $\|x_i - x_j\| \geq 1-\varepsilon$. It is easy to construct examples with many antipodes. However, one thing that these constructions seem to have in common is that they require that many pairs of points to be rather close to each other. Points are neighbors if $\|x_i -x_j\| \leq \varepsilon$.

\begin{center}
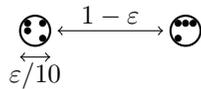
\begin{figure}[h!]
\begin{tikzpicture}
\draw [thick] (0,0) circle (0.2cm);
\draw [thick] (2,0) circle (0.2cm);
%\draw [thick] (2*0.5,2*0.86) circle (0.2cm);
\filldraw (-0.1, 0.1) circle (0.04cm);
\filldraw (0.1, 0.1) circle (0.04cm);
\filldraw (0.1, -0.1) circle (0.04cm);
\filldraw (-0.1, 0) circle (0.04cm);
\filldraw (2-0.1, 0.1) circle (0.04cm);
\filldraw (2.0, 0.1) circle (0.04cm);
\filldraw (2.1, 0.1) circle (0.04cm);
\filldraw (1.9, -0.1) circle (0.04cm);
%\filldraw [thick](2*0.5 +0.1,2*0.86) circle (0.04cm);
%\filldraw [thick](2*0.5 -0.1,2*0.86) circle (0.04cm);
%\filldraw[thick](2*0.5 +0.1,2*0.86+0.1) circle(0.04cm);
%\filldraw [thick](2*0.5 +0.1,2*0.86-0.1) circle (0.04cm);
\draw [<->] (-0.2, -0.35) -- (0.2, -0.35);
\node at (0, -0.6) {$\varepsilon/10$};
\draw [<->] (0.3, 0) -- (1.7, 0);
\node at (1, 0.2) {$1-\varepsilon$};
\end{tikzpicture}
\caption{$n$ points with $\sim n^2$ pairs of points at distance $\geq 1 - \varepsilon$ but also $\sim n^2$ pairs of points at distance $\leq \varepsilon$.}
\end{figure}
\end{center}
\vspace{-15pt}
\subsection{Result.} If there are many antipodes, then there are many neighbors.

\begin{thm} There exists a universal constant $c>0$ such that for all $\varepsilon >0$ and any set $\left\{x_1, \dots, x_n \right\} \subset \mathbb{R}^2$ with diameter $\leq 1$, i.e. satisfying $\|x_i -x_j\| \leq 1$,
   $$       \#\left\{ 1 \leq i,j \leq n: \|x_i - x_j\| \leq \varepsilon  \right\} \geq \frac{c \cdot \varepsilon^{3/4}}{\left( \log \varepsilon^{-1} \right)^{1/4}} \#\left\{ 1 \leq i,j \leq n: \|x_i - x_j\| \geq 1-\varepsilon \right\}.$$ 
\end{thm}
 A natural candidate for the optimal scaling is $\sim \sqrt{\varepsilon}$ and we present examples attaining that rate in \S 1.3. The problem is also interesting in higher dimensions and more general metric spaces and we comment on this in \S 1.4.

\subsection{Examples.}
We discuss three examples attaining a rate of $\sim \sqrt{\varepsilon}$. Now and henceforth, we do not track constants, $A \lesssim B$ means $A \leq cB$ for some absolute constant $c>0$ that is independent of everything. $A \gtrsim B$ is used in the same manner. $A \sim B$ means that both $A \gtrsim B$ and $A \lesssim B$ are valid at the same time.\\

\textit{1. Circle.} Pick a circle of radius $1/2$ and distribute $n$ points evenly along the circle. When $n \gg 1$ is sufficiently large (depending on $\varepsilon$), then each point has $\sim \sqrt{\varepsilon}\cdot n$ antipodes and $\sim \varepsilon \cdot n$ neighbors. Summing over all points implies the scaling relation $\#\mbox{neighbors} \sim \sqrt{\varepsilon} \cdot \#\mbox{antipodes}$.\\
\textit{2. Releaux Triangle.}
  Suppose $0 < \varepsilon \ll 1$, take the Reuleaux triangle and place $\sim \sqrt{\varepsilon} \cdot n$ points in each of the three vertices and distribute the remaining $\sim n$ points evenly across the three arcs. Each point in a vertex has $\sim n/3$ antipodes, each point on an arc is antipodal to $\sim \sqrt{\varepsilon} \cdot n$ points in the opposing vertex, there are $\sim \sqrt{\varepsilon}\cdot n^2$ antipodal pairs.  Regarding neighbors, we first note that the three clusters create $\sim 3 \varepsilon n^2$ neighbors. As for the remaining $\sim n$ points, they are spread evenly over three arcs of length $\sim 1$. Partitioning these into $\sim \varepsilon^{-1}$ boxes of size $\varepsilon \times \varepsilon$, we see that this creates a total of $\sim (\varepsilon n)^2 \varepsilon^{-1} = \varepsilon n^2$ neighbors for a total of $\sim \varepsilon n^2$ neighbors. Therefore $\#\mbox{neighbors} \sim \sqrt{\varepsilon} \cdot \#\mbox{antipodes}$.
  
\begin{figure}[h!]
    \begin{minipage}{.58\textwidth}
    \begin{tikzpicture}
         \tikz{\draw[thick] (90:1) arc (120:180:1.5) arc (-120:-60:1.5) arc (0:60:1.5);}
         \filldraw (-0.03,0.28) circle (0.05cm);
         \filldraw (-1.5,0.28) circle (0.05cm);
         \filldraw (-0.75,1.55) circle (0.05cm);
         \node at (0.3, 0.2) {$\sqrt{\varepsilon}n$ };
         \node at (-1.9, 0.2) {$\sqrt{\varepsilon}n$};
         \node at (-0.75, 1.8) {$\sqrt{\varepsilon}n$};
         \node at (3,0) {$\quad$};
\end{tikzpicture} 
\end{minipage}
\begin{minipage}{.38\textwidth}
\begin{tikzpicture}[scale=0.4]
    \def\n{10} % Number of sides
    \def\r{3}  % Radius of the circumcircle
    % Draw edges of the 11-gon
    \foreach \i in {0,...,9} {
        \draw[very thick] 
            ({\r*cos(360*\i/\n)}, {\r*sin(360*\i/\n)}) -- 
            ({\r*cos(360*(\i+1)/\n)}, {\r*sin(360*(\i+1)/\n)});
    }
    \filldraw (0,0) circle (0.04cm);
\draw [thick] (0,0) -- (0.9, 2.85);
\node at (1.1, 1.5) {$R$};
\draw [thick] (0,0) -- (0, 2.85);
\node at (-0.5, 1.6) {$a$};
\node at (0, 3.2) {$s$};
\node at (0,-5) {};
\end{tikzpicture}
\end{minipage}
\vspace{-10pt}
\caption{Left: a Reuleaux triangle with many points in the vertices, right: a regular polygon with points evenly spaced.}
\end{figure}
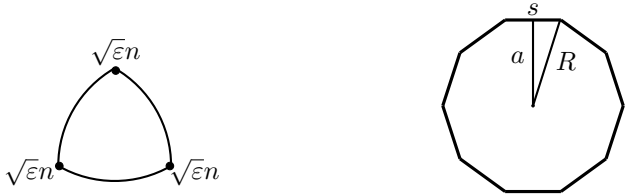

\textit{3. Regular polygon.} 
Maybe not all that different from the circle: take a regular $k-$gon with $k$ even. We want the polygon to have diameter 1, this forces $R=1/2$ (see Fig. 2). 
Then we have the relationships
$$ a = \frac{1}{2} \cos\left( \frac{\pi}{k} \right) = \frac{1}{2} - \frac{\pi^2}{4 k^2} + o(k^{-3})\qquad \mbox{and} \qquad s = \sin\left(\frac{\pi}{n}\right) = \frac{\pi}{k} + \mathcal{O}(k^{-2}).$$
In particular, when $2a \geq 1 - \varepsilon$, then each point on each side is an antipode with each point on the opposite side. This suggests setting $k$ so that  $\pi^2/(2 k^2) \sim \varepsilon$ and thus $k \sim \varepsilon^{-1/2}$. We spread $n \gg 1$ points evenly (with respect to arclength) on the boundary of the polygon. Then each point has at least $\sim n/k \sim \sqrt{\varepsilon}n$ antipodes, the points on the opposing side, leading to a total of $\sqrt{\varepsilon} n^2$ antipodes. Each point has $\sim \varepsilon n$ neighbors leading to a total of $\sim \varepsilon n^2$ neighbors. One might be tempted to observe that the first two examples are convex sets of constant width and the third is very nearly such a set; maybe all convex sets with constant width could be turned into such an example.

\subsection{Higher dimensions and metric spaces.} It would be interesting to understand what happens in higher dimensions. Our argument would also produce an exponent in $\mathbb{R}^d$, however, since the argument does not give the sharp rate in two dimensions, we only present the details in $\mathbb{R}^2$. Maybe the examples from \S 1.3 suggest a natural scaling in higher dimensions: using the Releaux triangle as motivation, consider placing points $\delta n$ in $\mathbf{0} \in \mathbb{R}^d$ and the distribute the remaining $n$ points as equispaced as possible (maximizing the minimal distance between any pair of points) over a spherical cap on the unit sphere $\mathbb{S}^{d-1} \subset \mathbb{R}^d$ whose radius is small enough so that the diameter does not exceed 1. Then, for $n$ sufficiently large depending only on $\varepsilon$ and the dimension, we have
$$  \# \left\{(i,j): \|x_i - x_j\| \leq \varepsilon\right\} \sim \delta^2 n^2 + \varepsilon^{d-1}n^2 $$
  while the number of antipodes is $\sim \delta n^2$. This motivates the choice $\delta \sim \varepsilon^{(d-1)/2}$ and leads to an example where $\#\mbox{neighbors} \sim \varepsilon^{(d-1)/2} \cdot \#\mbox{antipodes}$.
Alternatively, one may consider a sphere of radius $1/2$ and place $n$ points as evenly as possible over the sphere. When $n$ is very large (depending on $\varepsilon$), each point has antipodes in a spherical cap of diameter $\sim \sqrt{\varepsilon}$ centered in the opposite from the point (the classical use of the term `antipode'): such a cap has area $\sim \varepsilon^{(d-1)/2}$ and captures $\sim n\varepsilon^{(d-1)/2}$ points. For the same reason, each point has $\sim \varepsilon^{d-1} n$ neighbors. Once more we have $\#\mbox{neighbors} \sim \varepsilon^{(d-1)/2} \cdot \#\mbox{antipodes}$.
 It might be interesting to understand how the exponent relates to other properties of the underlying metric space.  For example, a star graph (see Fig. 3) suggests a metric space in which no such property holds:  consider a metric space on $\mathbb{N} = \left\{0,1,2,\dots\right\}$ such that
$d(0,n) = 1$ for all $n \geq 1$ while $d(i,j) = 2$ for all $i > j \geq 1$. There are arbitrarily large sets of points such that all pairs of points are maximally separated but no points are particularly close to each other. Seeing as this example is quite `hyperbolic', loosely speaking, one could wonder whether some type of generalized notion of curvature might come into play. 

\begin{center}
\begin{figure}[h!]
\begin{tikzpicture}
\node at (0,0) {\includegraphics[width=0.17\textwidth]{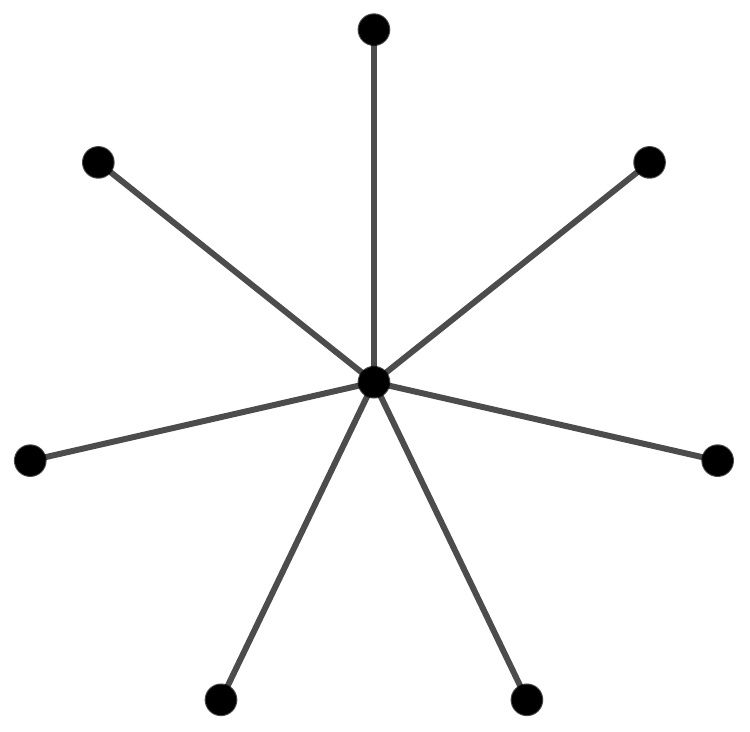}};
\end{tikzpicture}
\caption{A sketch of a type of metric space where no such property holds: there can be many antipodes, pairs of points at distance $\geq 2$ while any pair of points is at least $1-$separated.}
\end{figure}
\end{center}
\vspace{-10pt}

\subsection{Related results} We are not aware of any results in this direction. 
There are a number of results concerned with the problem of maximizing various notions of average distance between the points subject to a diameter constraint.
Witsenhausen \cite{wit} studied the problem of maximizing $\sum_{i,j} \|x_i-x_j\|^2$ subject to the points having diameter 1. Larcher-Pillichshammer \cite{larcher}, Pillichshammer \cite{fritz, fritz3} and Wolf \cite{wolf} considered $\sum_{i,j} \|x_i-x_j\|$. Pillichshammer \cite{fritz2} studied $\sum_{i,j} \|x_i-x_j\|^{\gamma}$ where $\gamma \geq 1.07$. For an unexpected connection of music theory, see 
Toussaint \cite{touss}. These results are at least partially motivated by earlier results of Fejes-T\'oth \cite{toth1, toth}. Our question seems vaguely related but the local-global interplay is somewhat different.

\section{Proof}

\subsection{Outline of the argument.}
The argument has the following steps.
\begin{enumerate}
    \item \S 2.2. It is enough to prove the statement for $\varepsilon < 1/1000$. This is an easy consequence of the fact that the number of neighbors is always $\gtrsim \varepsilon^2 n^2$ while the number of antipodes is always trivially $\lesssim n^2$. 
    \item \S 2.3. The second reduction allows us to assume that all the points are within distance $\leq \varepsilon$ of the boundary of the convex hull of the points.  Points that are further away cannot contribute to the antipode count while potentially counting as neighbors: erasing them makes the inequality stronger.
    \item \S 2.4. Once we are restricted to a $\varepsilon-$neighborhood of the boundary, we can partition this neighborhood into $\lesssim \varepsilon^{-1}$ sets of diameter $< \varepsilon/2$. Counting points in boxes leads to a bound in terms of a quadratic form. 
    \item \S 2.5. The problem is now estimating the largest eigenvalue $\lambda_1(M)$ of a matrix $M$. We use that $M$ is symmetric, $\lambda_1(M) = \sqrt{ \lambda_1(M^T M)}$, and  
    $$ \lambda_1(M^T M) \leq \sum_{i=1}^{n} \lambda_i(M^T M) = \mbox{tr}(M^T M).$$
    This reduces the problem to estimating $\mbox{tr}(M^T M)$ which will turn out to have a geometric/combinatorial significance. This is arguably the point where the argument is lossy, bounding an eigenvalue of a positive-definite matrix by the sum of all eigenvalues is typically quite wasteful.
    \item \S 2.6 discusses a Lemma from Graph Theory. 
    \item \S 2.7 concludes the argument by proving an inequality for $\mbox{tr}(M^T M)$ that is sharp up to a logarithmic factor.
\end{enumerate}

\subsection{Reducing to $\varepsilon$ small.}
It suffices to prove the inequality for $ 0 < \varepsilon < 1/1000$ or any other arbitrarily positive number as long as it is absolute. A set of diameter $\leq 1$ can be trivially placed in a disk of radius $2$. This is related to Lebesgue's universal covering problem for which much more precise results are known \cite{baez, hansen}. A disk of radius 2, in turn, can be partitioned into $ k \lesssim 1/\varepsilon^2$ boxes of diameter $< \varepsilon/2$ which we denote by $B_1, B_2, \dots, B_k$. Introducing $n_i = \# \left\{ 1 \leq i \leq n: x_i \in B_i \right\}$
to denote the number of points in $B_i$, we have with the Cauchy-Schwarz inequality 
$$  n = \sum_{i=1}^{k} n_i \leq \sqrt{k} \cdot \left( \sum_{i=1}^{k} n_i^2 \right)^{1/2}$$
from which we deduce, using $k \lesssim 1/\varepsilon^2$ and that the diameter of $B_i$ is $< \varepsilon /2$ (ensuring that all points in the same box are neighbors) that
$$ \varepsilon^2 n^2 \lesssim \frac{n^2}{k}  \lesssim \sum_{i=1}^{k} n_i^2 \lesssim    \#\left\{ 1 \leq i,j \leq n: \|x_i - x_j\| \leq \varepsilon  \right\}.$$
Since the number of antipodes is, trivially, $\lesssim n^2$, this implies that the inequality is true for any fixed $\varepsilon > 0$ with constant $ \sim \varepsilon^2$ and that we may assume $\varepsilon < 1/1000$.

\subsection{Reducing to a neighborhood of the convex hull.}  Consider the convex hull $\Omega = \mbox{conv}(X)$. Taking the convex hull of a set does not increase the diameter, the normalization $\diam(X) = 1$ implies that $\diam(\Omega) = 1$. The points $Y \subseteq X$ that are far from the boundary $\partial \Omega$, meaning
$$Y = \left\{x_i: d(x_i, \partial \Omega) > \varepsilon \right\},$$
cannot have distance $\geq 1-\varepsilon$ with any other point in the set and they cannot contribute to the number of antipodes. However, they can potentially increase the number of neighbors.  Deleting all the points in $Y$ leads to a stronger inequality and we may assume $Y = \emptyset$. A convex set with diameter $=1$ has circumference $\lesssim 1$ and we may thus partition the set
$$ \Omega_{\varepsilon} = \left\{ x \in \Omega: d(x, \Omega^c) \leq \varepsilon \right\}$$
into $k$ boxes of size $\varepsilon/4 \times \varepsilon/4$ that we again denote by $B_1, \dots, B_k$. 
\begin{center}
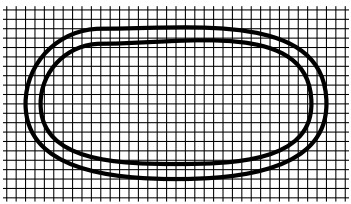
\begin{figure}[h!]
\begin{tikzpicture}
    \foreach \i in {-2,...,18} {
        \draw[] (-2.3, 0.125*\i) -- (2.3, 0.125*\i);
    }
        \foreach \i in {-18,...,18} {
        \draw[] (0.125*\i, -0.3) -- (0.125*\i, 2.3);
    }
    \draw [ultra thick] (0,0) to[out=0, in =270] (2, 1) to[out=90, in =0] (-1, 2) to[out=180, in =90] (-2, 1) to[out=270, in=180] (0,0);
        \draw [ultra thick] (0,0.2) to[out=0, in =270] (1.8, 1) to[out=90, in =0] (-1, 1.8) to[out=180, in =90] (-1.8, 1) to[out=270, in=180] (0,0.2);
\end{tikzpicture}    
\caption{Reducing to a $\varepsilon-$neighborhood of the boundary and then breaking it into $\varepsilon/4 \times \varepsilon/4$ boxes.}
\end{figure}
\end{center}

\begin{lemma}
    The number of boxes needed to cover $\Omega_{\varepsilon}$ is $k \sim 1/\varepsilon$.
\end{lemma}
\begin{proof}
    The lower bound is simple: the set $\Omega_{\varepsilon}$ is connected and has diameter 1. This requires at least $\sim 1/\varepsilon$ boxes of size $\varepsilon/4 \times \varepsilon/4$ to be covered. As for the upper bound, we suppose now that $\Omega_{\varepsilon}$ requires $k$ boxes of size $\varepsilon/4 \times \varepsilon/4$ to be covered. Then these boxes are all contained in a $3\varepsilon-$neighborhood of $\partial \Omega$ which is a set of area $\lesssim \varepsilon$ and the claim follows. 
\end{proof}

We note that this observation can be coupled with the one from above, where we showed that, with $k$ being the number of boxes,
$$ \frac{n^2}{k} \lesssim    \#\left\{ 1 \leq i,j \leq n: \|x_i - x_j\| \leq \varepsilon  \right\}.$$
Since the inequality improves when deleting points inside the convex domain, the boundary region can be
covered with $k \sim 1/\varepsilon$ boxes and since the number of antipodes is at most $\leq n^2$, this shows
$$ \# \left\{(i,j): \|x_i - x_j\| \leq \varepsilon\right\} \gtrsim \varepsilon \cdot \# \left\{(i,j): \|x_i - x_j\| \geq 1- \varepsilon\right\}.$$

\subsection{Linear Algebra.} The goal of this section is to rephrase the problem. Given boxes $B_1, \dots, B_k$, we introduce the matrix $M \in \mathbb{R}^{k \times k}$
$$ M = \left( \mathbf{1}_{\max_{x \in B_i, y \in B_j} \|x-y\| \geq 1-\varepsilon} \right)_{i,j=1}^{k}.$$
The matrix has entries $0$ and $1$ with $M_{ij} = 1$ meaning that it is possible for a point in $B_i$ and a point in $B_j$ to be antipodes. Introducing
$$ n_i = \# \left\{ 1 \leq j \leq n: x_j \in B_i \right\} $$
and the vector $\mathbf{n} = (n_1, n_2, \dots, n_k) \in \mathbb{N}^k$, one can bound the number of antipodes from above by
\begin{align*}
    \#\left\{ 1 \leq i,j \leq n: \|x_i - x_j\| \geq 1-\varepsilon \right\} \leq \sum_{i,j=1}^{k} n_i M_{ij} n_j =  \left\langle \mathbf{n}, M \mathbf{n}\right\rangle. 
\end{align*}    
Any two points in $B_i$ are going to be neighbors and thus
$$  \#\left\{ 1 \leq i,j \leq n: \|x_i - x_j\| \leq \varepsilon   \right\}  \geq \sum_{i=1}^{k} n_i^2 = \|\mathbf{n}\|^2.$$
Our desired result can be rephrased as
$$ \frac{\#\left\{ 1 \leq i,j \leq n: \|x_i - x_j\| \geq 1-\varepsilon \right\}}{  \#\left\{ 1 \leq i,j \leq n: \|x_i - x_j\| \leq \varepsilon  \right\}} \lesssim  \frac{\left( \log \varepsilon^{-1} \right)^{1/4}}{\varepsilon^{3/4}} $$
and the bounds above imply that
$$  \frac{\#\left\{ 1 \leq i,j \leq n: \|x_i - x_j\| \geq 1-\varepsilon \right\}}{  \#\left\{ 1 \leq i,j \leq n: \|x_i - x_j\| \leq \varepsilon  \right\}} \leq \frac{ \left\langle \mathbf{n}, M \mathbf{n}\right\rangle}{\|\mathbf{n}\|^2}.$$
 The remainder of the argument shows $\left\langle \mathbf{n}, M \mathbf{n}\right\rangle/\| \mathbf{n}\|^2 \lesssim  \left( \log \varepsilon^{-1} \right)^{1/4}\varepsilon^{-3/4}.$

\subsection{Linear Algebra} At this point, we use techniques from (elementary) linear algebra. $M$ is symmetric and the Courant-Fischer Theorem implies
$$ \frac{ \left\langle \mathbf{n}, M \mathbf{n}\right\rangle}{\|\mathbf{n}\|^2} \leq \sup_{\mathbf{x} \neq 0} \frac{ \left\langle \mathbf{x}, M \mathbf{x}\right\rangle}{\|\mathbf{x}\|^2} = \lambda_1(M),$$
where $\lambda_1(M)$ is the largest eigenvalue of $M$. Since $M$ is symmetric
$$ \lambda_1(M) =  \sqrt{ \lambda_1(M^T M)}.$$
The matrix $M^T M$ is symmetric and positive-definite
and all its eigenvalues are nonnegative. As a consequence
$$  \lambda_1(M^T M)   \leq \sum_{i=1}^{k} \lambda_i(M^T M) = \mbox{tr}(M^T M).$$
This is presumably where the argument stops being tight. The desired result follows from
$$ \mbox{tr}(M^T M) \lesssim  \frac{\left( \log \varepsilon^{-1} \right)^{1/2}}{\varepsilon^{3/2}}.$$
This inequality has a purely combinatorial interpretation since
\begin{align*}
     \mbox{tr}(M^T M) &= \sum_{i=1}^{k} (M^T M)_{ii} = \sum_{i=1}^{k} \sum_{\ell =1}^{k} (M^T)_{i \ell}M_{\ell i} = \sum_{i=1}^{k} \sum_{\ell =1}^{k}  M_{\ell i}^2 = \sum_{i=1}^{k} \sum_{\ell =1}^{k}  M_{\ell i} \\
     &= \# \left\{ (i,j) \in \left\{1,2,\dots, k\right\}^2:  \max_{x \in B_i, y \in B_j} \|x-y\| \geq 1-\varepsilon \right\}.
\end{align*}
 Since $k \sim 1/\varepsilon$, it remains to show that this sum is $\lesssim \sqrt{\log{k}} \cdot k^{3/2}$.

\subsection{Some Graph Theory}
We may think of the $k$ boxes as $k$ vertices in a graph. We then add edges connecting the vertices $i$ and $j$ whenever the condition $\max_{x \in B_i, y \in B_j} \|x-y\| \geq 1-\varepsilon$ is satisfied. It remains to give a bound on the  number of edges. For the sake of exposition, we state the combinatorial statement and prove it. \S 2.7 shows that it is applicable. Given a vertex $v \in V$, we use $N(v) = \left\{w \in V: (v,w) \in E\right\}$ to denote its neighborhood.

\begin{lemma}
    Let $G=(V,E)$ be a graph on $k$ vertices. Let $c>0$ be arbitrary. Assume that for each vertex $v \in V$ there exists a set of vertices $N_v \subset V$ with $|N_v| \leq c \sqrt{k}$ about which we know nothing. Then, for every $1 \leq s \leq k$, the number of vertices from $V \setminus N_v$ that have a lot of common neighbors with $v \in V$ is small
    $$ \left|\left\{ w \in V \setminus N_v: |N(v) \cap N(w)| \geq s \right\} \right| \leq c\frac{k}{s}.$$
    Then $|E|$ is bounded, with an implicit constant only depending only on $c$, by
    $$ |E|\lesssim  \sqrt{\log{k}} \cdot k^{3/2}.$$
\end{lemma}

\begin{proof}
We start with the Cauchy-Schwarz inequality and double counting,
 $$ \left( \sum_{v \in V} \deg(v) \right)^2 \leq k   \sum_{v \in V} \deg(v)^2 = k  \sum_{a,b \in V} |N(a) \cap N(b)|.$$
For any fixed $a \in V$, we have
\begin{align*}
    \sum_{b \in V} |N(a) \cap N(b)| &= \sum_{b \in N_a} |N(a) \cap N(b)| +  \sum_{b \in V \setminus N_a} |N(a) \cap N(b)|  \\
    &\leq \sum_{b \in N_a} | N(a)| + \sum_{b \in V \setminus N_a} |N(a) \cap N(b)|.
\end{align*}
The first sum is easy to bound, the size of the neighborhood is the degree $|N(a)| = \deg(a)$ which we sum over $|N_a| \leq c\sqrt{k}$ times for a total contribution of $\leq c\sqrt{k} \cdot \deg(a)$.
We bound the second sum using the assumption and  
\begin{align*}
      \sum_{b \in V \setminus N_a} |N(a) \cap N(b)| = \sum_{\ell=1}^{k}  \left|\left\{ w \in V \setminus N_a: |N(a) \cap N(w)| \geq \ell \right\} \right| \leq c\sum_{\ell=1}^{k} \frac{k}{\ell}
\end{align*}
giving a contribution of $\leq 2 c k \log{k}$. 
Summing over $a \in V$, we arrive at
\begin{align*}
  \left( \sum_{v \in V} \deg(v) \right)^2 &\leq k \sum_{a \in V} \left(c \sqrt{k} \cdot \deg(a) + c 2 k \log{k} \right) \\
  &\leq c k^{3/2}   \left( \sum_{v \in V} \deg(v) \right) + c 2 k^3 \log{k}.
\end{align*}
which implies $|E| \lesssim \sqrt{c \log{k}} \cdot k^{3/2} $.
\end{proof}

\subsection{Conclusion of the Argument} We now establish the following Lemma.
\begin{lemma} We have
$$ \sum_{i=1}^{k} \# \left\{ 1 \leq j \leq n:  \max_{x \in B_i, y \in B_j} d(x,y) \geq 1-\varepsilon  \right\} \lesssim \sqrt{\log{k}} \cdot k^{3/2} .$$
\end{lemma}
This is sharp up to the logarithmic factor. Using the regular polygon with $\sim 1/\sqrt{\varepsilon}$ sides, chosen to be an even number, we see that each line segment is made up of $\sim 1/\sqrt{\varepsilon}$ boxes. Each box is antipodal to all the boxes on the `antipodal' side, which creates $\sim 1/\sqrt{\varepsilon} \times 1/\varepsilon = \varepsilon^{-3/2} \sim k^{3/2}$ antipodal boxes.

\begin{proof}
We use the Lemma from \S 2.6.  Boxes are vertices. The convexity of $\Omega$ ensures that any box $B_i$ has the property that there are $\lesssim \varepsilon^{-1} d$ boxes within distance $\leq d$. We think of any box $B_j$ within distance $\lesssim \sqrt{\varepsilon}$ of $B_i$ as `forbidden' and make no statement about it. 
Suppose now that  
 $$  d = \min_{x \in B_i, y \in B_j} d(x,y) \geq 10 \sqrt{\varepsilon}.$$
We want to show that $B_i$ and $B_j$ have few common neighbors (depending on $d$).
Pick $a \in B_i$ and $b \in B_j$ chosen so as to minimize the distance. The problem is invariant under translation and rotation so we may assume that both $a$ and $b$ are on the $x-$axis and $a = -d/2$ while $b=d/2$. The points that are distance $\geq 1-\varepsilon$ from both $a$ and $b$ are contained in the intersection of two annuli
$$ \left\{x \in \mathbb{R}^2: 1 -\varepsilon \leq \|x-a\| \leq 1 \right\} \cap \left\{x \in \mathbb{R}^2: 1 -\varepsilon \leq \|x-b\| \leq 1 \right\}.$$

\begin{center}
    \begin{figure}[h!]
\begin{tikzpicture}
    \filldraw (0.95,0) circle (0.07cm);
    \node at (-0.78, -0.4) {$a$};
        \filldraw (-0.78,0) circle (0.07cm);
     \node at (0.95, -0.4) {$b$};
     \draw[thick] (-1,0.12) circle (0.3cm);  
     \draw [thick] (0.95, 0) to[out=270, in=180] (1.1, -0.2) to[out=0, in=270] (1.3, 0.1) to[out=90, in=0] (1, 0.2) to[out=180, in=90] (0.95,0);
     \node at (-1.5, 0.5) {$B_i$};
   \node at (1.5, 0) {$B_j$};
      \draw [thick, domain=45:100] plot ({-1+4*cos(\x)}, { 4*sin(\x)});
     \draw [thick, domain=45:100] plot ({-1+3.7*cos(\x)}, { + 3.7*sin(\x)});
     \draw [thick, domain=80:135] plot ({+1+3.7*cos(\x)}, { + 3.7*sin(\x)});
     \draw [thick, domain=80:135] plot ({+1+4*cos(\x)}, { + 4*sin(\x)});
     \draw [<->, thick] (-0.6, 0) -- (0.7, 0);
     \node at (0, -0.3) {$d$};
     \filldraw (0.58, 3.67) circle (0.05cm);
       \filldraw (-0.58, 3.67) circle (0.05cm);
         \filldraw (0, 3.56) circle (0.05cm);
       \filldraw (0, 3.88) circle (0.05cm);   
\end{tikzpicture}
\caption{A cartoon picture of the argument.}
    \end{figure}
\end{center}
\vspace{-10pt}

We refer to Fig. 5 for a sketch -- note that the intersection of two annuli could be comprised of two connected components: however, since $a$ and $b$ are distance at most 1, we see from the equilateral triangle that the distance
of these two connected components is at least $\geq 2 (\sqrt{3}/2) \gg 1$ and thus they cannot both be in the set. There are four intersection points and solving the four quadratic equations determines where these points are: the two points on the $y-$axis are
$$ \left(0, \frac{\sqrt{4 - d^2}}{2} \right) \quad \mbox{and} \quad \left(0, \frac{ \sqrt{4 -d^2+4 e^2-8 e}}{2} \right).$$
The difference between these two points is order $\sim \varepsilon$. The other two points are symmetric around the $y-$axis and given by
$$ \left( \pm \frac{2\varepsilon - \varepsilon^2}{2d}, \frac{\sqrt{-d^4+2 d^2 e^2-4 d^2 e+4 d^2-e^4+4 e^3-4 e^2}}{2 d} \right)$$
and the distance between these points is $\lesssim \varepsilon/d$. This spherical quadrilateral can be covered with $\lesssim 1/d$ boxes of size $\varepsilon/4 \times \varepsilon/4$. Moreover, if we now replace $a$ and $b$ by other points in $B_i$ and $B_j$, the sets cannot move more than distance $<\varepsilon/4$ trapping us in a $\varepsilon-$neighborhood of the quadrilateral which does not change the dimensions by more than a constant. Thus the set of all points that can be potentially antipodal to points from both $B_i$ and $B_j$ can be covered with $\lesssim 1/d$ boxes of size $\varepsilon/4 \times \varepsilon/4$.  This means that if a box $B_j$ were to have $\geq s$ antipodal neighbors with $B_i$, then we require $d \lesssim 1/s$ and because of convexity of the boundary there are $\lesssim \varepsilon^{-1}/s$ boxes in that neighborhood. Since $\varepsilon^{-1} \sim k$, the Lemma applies.

\end{proof}

\end{document}